\documentclass[a4paper
,11pt
,parskip=half
,appendixprefix=true
]{amsart}
\usepackage{libertine}
\usepackage{amssymb,latexsym}
\usepackage[T1]{fontenc}
\usepackage{amsmath}
\usepackage{amsfonts}
\usepackage{amsthm}
\usepackage{enumerate}
\usepackage{ifthen}
\usepackage{amssymb}
\usepackage{comment}
\usepackage{xcolor}
\usepackage{booktabs}

\title[Fourth Moment Theorems]
{Fourth Moment Theorems for Markov Diffusion Generators}
\author[Azmoodeh, Campese and Poly]{Ehsan Azmoodeh, Simon Campese  \and
  Guillaume Poly}
\date{\today}
\address{Universit\'e du Luxembourg\\ Facult\'e des Sciences, de la Technologie
  et de la Communication\\ Unit\'e de Recherche en Math\'emathiques\\ \\ 6, rue
  Coudenhove Kalergi \\ 1359 LUXEMBOURG}
\email{ehsan.azmoodeh@uni.lu, simon.campese@uni.lu, guillaume.poly@uni.lu}

\keywords{Markov operator, diffusion generator, eigenfunction, Chaos, Fourth moment theorem}
\subjclass[2010]{60F05, 60J35, 60J60, 60J99, 60H99}

\begin{document}
\begin{abstract}
Inspired by the insightful article~\cite{ledoux_chaos_2012}, we revisit the
Nualart-Peccati-criterion~\cite{nualart_central_2005} (now known as the Fourth
Moment Theorem) from the
point of view of spectral theory of general Markov diffusion generators. We are
not only able to drastically simplify all of its previous proofs, but also to
provide new settings of diffusive generators (Laguerre, Jacobi) where such a
criterion holds. Convergence towards gamma and beta distributions under moment
conditions is also discussed. 
\end{abstract}
\maketitle
\theoremstyle{plain}
\newtheorem{theorem}{Theorem}[section]
\newtheorem{proposition}[theorem]{Proposition}
\newtheorem{lemma}[theorem]{Lemma}
\newtheorem{corollary}[theorem]{Corollary}
\newtheorem*{thm:mix}{Theorem \ref{thm:mix}}

\theoremstyle{definition}
\newtheorem{definition}[theorem]{Definition}

\theoremstyle{remark}
\newtheorem*{remark}{Remark}
\newtheorem*{example}{Example}

\renewcommand{\thesection}{\arabic{section}}

\newcommand{\mi}{\ensuremath{\mathrm{i}}}
\newcommand{\me}{\ensuremath{\mathrm{e}}}
\newcommand{\abs}[1]{\left| #1 \right|}
\newcommand{\id}{\operatorname{I}}

\newcommand{\diff}[1]{\operatorname{d}\ifthenelse{\equal{#1}{}}{\,}{\!#1}}
\newcommand{\E}{\mathbb{E}}
\newcommand{\charop}{\operatorname{\mathbb{L}}}
\newcommand{\on}[1]{\operatorname{#1}}
\newcommand{\proj}[2]{\operatorname{proj}\left( #1 \mid E_{#2} \right)}

\newcommand{\C}{\mathbb{C}}
\newcommand{\R}{\mathbb{R}}
\newcommand{\Q}{\mathbb{Q}}
\newcommand{\N}{\mathbb{N}}

\newcommand{\ud}{\mathrm{d}}
\newcommand{\LL}{\mathrm{\textbf{L}}}
\newcommand{\Id}{\mathrm{\textbf{Id}}}
\newcommand{\K}{\mathrm{\textbf{Ker}}}


\tableofcontents

\section{Introduction}
In 2005, Nualart and Peccati~\cite{nualart_central_2005} discovered the
surprising fact that any sequence of random variables $\{X_n\}_{n\ge 1}$ in a
Gaussian chaos of fixed order converges in distribution towards a standard
Gaussian random variable if and only if $\mathbb{E}(X_n^2)\to 1$ and $
\mathbb{E}(X_n^4)\to 3$. In fact, this result contains the two following
important informations of a different nature:

\begin{itemize}
\item[(i)]For all non-zero $X$ in a Wiener chaos with order $\ge 2$,
  $\E(X^4)>3\E(X^2)^2$, 
\item[(ii)]$\E(X^4)-3\E(X^2)^2\approx 0$ if and only if $X
  \stackrel{Law}{\approx}\mathcal{N}(0,\E(X^2))$. 
\end{itemize}
~\\
This striking discovery, now known as the Fourth Moment Theorem, has been the
starting point of a fruitful line of 
research of which we shall give a quick overview. The proof of the above result
given in~\cite{nualart_central_2005} used the Dambis-Dubins-Schwartz Theorem
(see e.g. \cite[ch. 5]{revuz_continuous_1999}) and did not provide any
estimates. In~\cite{nualart_central_2008}, this phenomenon was translated in
terms  
of Malliavin operators, whereas in \cite{nourdin_steins_2009} these operators
were combined with the bounds arising from Stein's method, thus yielding both a
short proof and precise estimates in the total variation distance (see
also~\cite{nourdin2013optimal}). The 
main difficulty of the proof consists of establishing the powerful inequality
\begin{equation}\label{ineg-princip}
  \operatorname{Var}\big{(}\Gamma(X)\big{)}\leq C \, \big{(}\E(X^4)-3
  \E(X^2)^2\big{)}, 
\end{equation}
where $\Gamma$ is the carr\'{e} du champ operator associated with the
generator of the Ornstein-Uhlenbeck semigroup, from which one can almost
immediately deduce convergence in law towards a standard 
Gaussian distribution $\mathcal{N}(0,1)$ (for instance by Stein's lemma). Other
proofs of the Fourth Moment Theorem, not necessarily relying on the
inequality~\eqref{ineg-princip}, can be found in~\cite{nourdin_yet_2011},
\cite{nourdin_asymptotic_2011} and \cite{ledoux_chaos_2012}.
We also mention~\cite{Speicher-nourdin} for extensions to the free
probability setting, \cite{peccati_gaussian_2005} for the multivariate setting,  
\cite{nourdin_noncentral_2009} for Gamma approximation, as well as \cite{utzet}
for the discrete setting.
It is important to note that virtually all the proofs (with the remarkable
exception of~\cite{ledoux_chaos_2012}), make crucial use of the product formula
for multiple integrals and thus rely on a very rigid structure of the underlying
probability space. As a matter of fact, this approach does not cover other
important structures like Laguerre and Jacobi, which are investigated in the
present article.

In the recent article~\cite{ledoux_chaos_2012}, M. Ledoux gave another proof of
the Fourth Moment Theorem in the general framework of diffusive Markov
generators, adopting a purely spectral point of view. In particular, he
completely avoids the use of product formulae. Unlike the Wiener space setting,
it turns out that in this more general framework it is not sufficient anymore
that a random variable $X$ is only an eigenfunction of the diffusion 
generator for an equality of the type~\eqref{ineg-princip} to hold. By imposing
additional assumptions, one is thus naturally led to a general definition of
chaos. We emphasize that this starting definition of chaos is the cornerstone of
the whole strategy. The definition of "general chaos" given
in~\cite{ledoux_chaos_2012} makes use of iterated gradients and, although
including the Ornstein-Uhlenbeck case, prevents one to reach the Laguerre and  
Jacobi structures. In this article, we keep the insightful idea
from~\cite{ledoux_chaos_2012} of encoding the Fourth Moment Theorem in purely
spectral form, but at the very beginning generalize once more the notion of
"general chaos". As we  
shall see, we say that $X \in \K \left( \LL + \lambda_p \Id \right)$ is a chaos
eigenfunction with respect to a Markov generator $\LL$ with spectrum $\{
\lambda_n \}_{n \geq 0}$, if and only if  
\begin{equation}\label{defi-chaos} 
X^2\in\bigoplus_{\alpha\leq \lambda_{2p}} \K \left( \LL + \alpha \Id \right).
\end{equation}

This definition has the following main advantages.

\begin{itemize}
\item Our definition of chaos covers the definition in~\cite{ledoux_chaos_2012},
    i.e. being a chaos in the sense of Ledoux
    implies~\eqref{defi-chaos}. Besides~\eqref{defi-chaos} seems easier to check
    in practice (see the remark after Theorem~\ref{theorem:2}).
\item We are able to deduce almost immediately that for any $X$ which is a chaos 
  eigenfunction in the above sense we have:
  \begin{equation*}
    \operatorname{Var}(\Gamma(X))\leq C \, \big( \E (X^4)-3\E (X^2)^2 \big).
  \end{equation*}
This drastically simplifies all the known proofs of the Fourth Moment Theorem
which are all based on the above inequality. 
\item
  With our definition of chaos,  we can extend the Fourth Moment Theorem to
    eigenfunctions of the Laguerre generator with any parameters. We mention 
    that, due to links between Hermite polynomials and Laguerre polynomial  
with integer parameters, the chaos of the Laguerre structure with integer
parameters can be plugged into the Wiener chaos. However, for non integer
parameters this case is uncovered by the existing literature and provides a new
framework where the Fourth  
Moment Theorems holds. As an illustrative example of the efficiency of our 
method, we can also mix the structures of Wiener and Laguerre to obtain
theorems of the following nature (a detailed proof will be given in Section~\ref{s-here-we-explain}):
\begin{theorem}\label{thm:mix}
  Let
$\mathcal{X}=\{\pi_{1,\nu}-\nu\,|\,\nu>-1\}\cup\{\mathcal{N}(0,1)\},$
be the set containing all  centered gamma and Gaussian laws.
Let  $\{X_i\}_{i\geq 1}$ be a sequence of independent random variables such that
for all $i\geq 1$, the law of $X_i$ belongs to $\mathcal{X}$. Now choose $d\geq
1$ and let 
\begin{equation*}
  P_n(x)=\sum_{i_1<i_2<\cdots<i_d}
a_n(i_1,\cdots,i_d) x_{i_1}\cdots x_{i_d}
\end{equation*}
be a sequence of multivariate
homogeneous polynomials of degree $d$. Then the two following statements are
equivalent:  
\begin{enumerate}[(i)]
\item  As $n$ tends to infinity, the sequence $\{ P_n(X) \}_{n \ge 1}$ converges
  in distribution towards $\mathcal{N}(0,1)$. 
\item As $n$ tends to infinity, it holds that $\E(P_n(X)^2)\to 1$ and
  $\E(P_n(X)^4)\to 3$. 
\end{enumerate}
\end{theorem}
We stress that due to the rather complicated structure of the variables (with no
symmetry), such a result seems hardly reachable by using product formulae.
\item We are also able to deal awith beta approximation, i.e. to
    provide conditions on the moments of a sequence $X_n$ of chaotic
    eigenfunctions under which the latter converges in distribution towards a 
    Beta  
    distribution. As we shall see, the only restriction is that the parameters
    $\alpha$ and $\beta$ have to satisfy the inequality $\alpha+\beta \leq
    1$. However, many important special cases like the Arcsine law are
    covered. Moreover our results provide new differential  
inequalities for the beta distributions (in the spirit of
inequality~\eqref{ineg-princip}) which are of independent interest.
\item Finally, we would like to mention that our strategy enables us to extend
  the Nualart-Peccati criterion to other couples of moments than $2$ and 
  $4$. For instance, one can prove that if $\E (X_n^4)\to 3$ and $\E(X_n^6) \to
  15$, then the
Central Limit Theorem holds. Up to now, this has remained an open 
question. Nevertheless, this is topic is under current research and will be
published in a forthcoming article.
\item Due to the simplicity of our proofs which are only of  spectral nature, we
  believe that our strategy can also be applied in the free and discrete
  settings (see \cite{Speicher-nourdin} for Wigner and~\cite{utzet} for Poisson
  chaos). 
\end{itemize}

\section{Main Results}
\label{s-subs-princ-thro}

\subsection{General principle}

Throughout the whole paper, we adopt the setting introduced
in~\cite{ledoux_chaos_2012}. Thus, we fix a probability space
$(E,\mathcal{F},\mu)$ and 
a symmetric Markov generator $-\LL$ with state space $E$ and probability measure
$\mu$ as its  
invariant measure. We assume that $-\LL$ has discrete spectrum $S = \{ \lambda_k
\}_{k \geq 0}$ and order its eigenvalues by magnitude, i.e. $0 = \lambda_0 <
\lambda_1 < \lambda_2 < \dots$. In other words, $-\LL$ is a self-adjoint, linear
operator on $L^2(E,\mu)$ with the property that $- \LL (1) = 0$.
By our assumption on the spectrum, $\LL$ is diagonalizable and we have
that
\begin{equation*}
  L^2(E,\mu)=\bigoplus_{k=0}^\infty \K(\LL+\lambda_k\Id).
\end{equation*}
The orthogonal projection of $X \in L^2(E,\mu)$ on the eigenspace $\K \left( \LL +
  \lambda_k \Id \right)$ will be denoted by $J_k(X)$.
Furthermore, we
define the associated bilinear carr\'e du champ operator $\Gamma$ by
\begin{equation*}
   \Gamma(X,Y) = \frac{1}{2} \left( \LL (XY) - X \LL Y - Y \LL X \right).
\end{equation*}
 If both arguments coincide, we
write $\Gamma(X)$ instead of $\Gamma(X,X)$. It follows from the definition of
$\Gamma$, that for any $X,Y \in L^2(E,\mu)$ the integration by parts formula 
\begin{equation}
  \label{eq:19}
  \int_E^{} \Gamma(X,Y) \diff{\mu} = - \int_E^{} X \LL Y \diff{\mu}=-\int_E^{} Y \LL X \diff{\mu}
\end{equation}
holds. For further details on this setting, we refer
to~\cite{bakry_hypercontractivite_1994} and the forthcoming 
book~\cite{bakry_gentil_ledoux_2013}. 

The following Theorem is the starting point of our investigations.
\begin{theorem}
  {\label{General Principle}}
In the above setting, 
let $\{ X_n\}_{n \geq 1}$ be a sequence in $L^2(E,\mu)$ such that each $X_n$ lies
in a common finite sum of eigenspaces of $\LL$, i.e. there exists $p \geq 0$ such that 
\begin{equation*}
X_n\in\bigoplus_{k=0}^p \K(\LL+\lambda_k \Id)
\end{equation*}
for all $n \in \N$.
Then it holds for any $\eta \geq \lambda_p$ that
\begin{equation}
  \label{eq:15}
  \int_E X_n
  \left( \LL + \eta \Id \right)^2 X_n \diff{\mu}
  \leq
  \eta
  \int_E^{} X_n \left( \LL + \eta \Id \right) X_n \diff{\mu}
  \leq
  c
  \int_E X_n
  \left( \LL + \eta \Id \right)^2 X_n \diff{\mu},
\end{equation}
where $1/c$ is the minimum of the set $\left\{ \eta - \lambda_k \mid 0 \leq k
  \leq p \right\} \setminus \{ 0 \}$.
In particular, the following two conditions are equivalent.
\begin{enumerate}[(i)]
 \item As $n$ tends to infinity, it holds that
   $\int_E X_n
      \left( \LL + \eta \Id \right)^2 X_n \diff{\mu} \to 0$.
\item As $n$ tends to infinity, it holds that 
    $\int_E^{} X_n \left( \LL + \eta \Id \right) X_n \diff{\mu}
    \to 0$.
\end{enumerate}
\end{theorem}

\begin{proof}
It holds that
  \begin{align*}
  \int_E^{} X_n  \left( \LL + \eta \Id \right)^2 X_n 
  \diff{\mu}
  &=
  \int_E^{}
  X_n  \LL \left( \LL + \eta \Id \right)X_n \diff{\mu}
  + \eta
  \int_E^{}
  X_n (\LL +\eta \Id) X_n \diff{\mu}
  \\ &=
  \sum_{k=0}^p (-\lambda_k)(\eta -\lambda_k) \int_E^{} J_k(X_n)^2
  \diff{\mu}
  +
  \eta \int_E^{} X_n \left( \LL + \eta \Id \right) X_n \diff{\mu}
  \\ &\leq
  \eta \int_E^{} X_n \left( \LL + \eta \Id \right)X_n \diff{\mu}
  \end{align*}
  and
  \begin{align*}
    \int_E^{} X_n \left( \LL + \eta \Id \right) X_n \diff{\mu}
    &=
    \sum_{k=0}^{p} (\eta - \lambda_k) \int_E^{} J_k(X_n)^2 \diff{\mu}
    \\ &\leq
    c
    \sum_{k=0}^{p} (\eta-\lambda_k)^2
    \int_E^{} J_k(X_n)^2 \diff{\mu}
    \\ &=
    c \int_E^{} X_n \left( \LL + \eta \Id
    \right)^2 X_n 
    \diff{\mu}.
  \end{align*}
\end{proof}

\begin{remark} 
We will see later that being able to remove the square from the operator 
$\left( \LL +\lambda_p \Id \right)^2$ will prove
itself very useful in both abstract and concrete frameworks. To see the latter,
note that if for example $E = \R^d$ and $\LL$ is a diffusion generator, then
$\left( \LL + \lambda_p \Id \right)^2$ is a differential operator of order four
while the non-squared version only has order two. For example, we are able to
give a drastically simplified proof of the classical Fourth Moment Theorem
of~\cite{nualart_central_2005}. 
\end{remark}

\subsection{Chaos of a Markov generator}

As indicated in the introduction and further detailed below, for a
sequence $(X_n)_{n \geq 1}$ of eigenfunctions of $\LL$, convergence in law
towards many target measures is implied by $L^2(E,\mu)$-convergence of an
expression of the form $\Gamma(X_n) - P(X_n)$ towards zero, where $P(x)$ is
some  polynomial with degree at most two. As by definition $2 \Gamma(X_n) = \left(
  \LL + 2\lambda_p \Id \right) X_n^2$, Theorem~\ref{General Principle} suggests the
following definition of chaos. 

\begin{definition}\label{chaos} An eigenfunction $X$ of the generator
  $-\LL$ with eigenvalue $\lambda_p$ is called a \emph{chaos eigenfunction of order
  $p$}, if and only
  if
  \begin{equation}
    \label{eq:3}
    X^2 \in \bigoplus_{k=0}^{2p} \K(\LL+\lambda_k\Id).
  \end{equation}
\end{definition}

\begin{remark} 
   It is not clear if the $p$th chaos, i.e. the set of all chaos
    eigenfunctions of order $p$, is always a linear subspace of $\K \left( \LL
      +\lambda_p \Id \right)$. Indeed, there is no reason for the product $XY$
    of two chaos eigenfunctions of order $p$ to have an expansion of the
    form~\eqref{eq:3}. However, in many important examples all eigenfunctions
    are chaotic. For example, this phenomenon occurs if the
    eigenfunctions can be represented in terms of multivariate polynomials, as
    is always
    the case in the three most important diffusion structures, namely Wiener,
    Laguerre and Jacobi (see section~\ref{Appli}) and also in discrete settings
    like the Poisson space.
\end{remark}

\begin{remark}[Connection with Ledoux's definition and product fromulae]
  By writing~\eqref{eq:3} in the equivalent 
form 
\begin{equation}
  \label{eq:12}
  \tag{\ref{eq:3}'}
  X^2 = \sum_{k=0}^{2p} J_k \left( X^2 \right),
\end{equation}
one recovers an abstract version of the product formula. As indicated in the
introduction, explicit versions of this formula have been a crucial tool in
virtually all classical proofs of the Fourth Moment Theorem but are not needed
for the method presented here. Note also that the projection $J_k \left( X^2  
\right)$ on $\K \left( 
  \LL + 
  \lambda_k \Id \right)$ is explicitly given by
\begin{equation*}
  J_k \left( X^2 \right)
  =
  \left(
    \prod_{\substack{1 \leq i \leq 2p \\ i \neq k}}^{}
    (\lambda_i - \lambda_k)^{-1} \left( \LL + \lambda_i \Id \right)
  \right) X^2.
\end{equation*}
In the same spirit, we could state~\eqref{eq:3} in yet another equivalent way,
namely 
\begin{equation}
  \label{eq:13}
  \tag{\ref{eq:3}''}
  \left(
    \prod_{i = 0}^{2p}
    \left( \LL + \lambda_i \Id \right)
  \right) X^2
  = 0.
\end{equation}
In this form, through the identity $2 \Gamma(X) = \left( \LL + 2\lambda_p \Id \right) X^2$,
valid for any eigenfunction of $-\LL$ with eigenvalue $\lambda_p$, we see that
Ledoux's definition of chaos in~\cite{ledoux_chaos_2012} is a special case of  
ours. Indeed, his condition $Q_p(\Gamma)(X) = 0$ (see the original article
for a definition of the polynomial $Q_p$ and the operator $Q_p(\Gamma)$) is
equivalent to 
\begin{equation*}
  \left(
  \prod_{i=0}^{p} \left( \LL + 2 \left( \lambda_p - \lambda_i \right) \right)
\right)
X^2 = 0.
\end{equation*}
The restriction that only even eigenspaces are allowed in the expansion of
$X^2$ and that $2 (\lambda_p-\lambda_i)$ does not neccessarily have to lie in
the spectrum of $-\LL$ are lifted by Definition~\ref{chaos}. In particular,
eigenfunctions of the Laguerre and Jacobi generator, which (except for trivial
cases) do not satsify Ledoux's definition of chaos, are always chaotic in our
sense (see 
Section~\ref{Appli}).  
\end{remark}

\section{Fourth Moment Theorems for Diffusion Generators}
\label{s-subs-appr-let}

Still retaining the setting introduced in the previous section, we now and until
the end of this article assume
that the generator $\LL$ is diffusive, i.e. that for any test function
$\phi \in \mathcal{C}^{\infty}(\R)$ and any $X \in L^2(E,\mu)$ it holds that
\begin{equation*}
  \LL \phi(X) = \phi'(X) \LL X + \phi''(X) \Gamma(X).
\end{equation*}
Equivalently, $\Gamma$ is a derivation, in the sense that $\Gamma \left(
  \phi(X),X \right) = \phi'(X) \Gamma(X)$. We also need the technical assumption
that the eigenspaces are hypercontractive
(see~\cite{bakry_hypercontractivite_1994} for sufficient
conditions). 
We will give Fourth Moment Theorems
for convergence of a sequence $\{X_n\}_{n \geq 1}$ of eigenfunctions of $\LL$
towards a Gaussian, Gamma or Beta distribution. When seen as invariant
measures of another diffusive and symmetric Markov generator $\mathcal{L}$ on
$\R$ with discrete spectrum, these three measures are 
the only ones that can arise if one assumes that the eigenfunctions of
$\mathcal{L}$ are orthogonal polynomials (see~\cite{mazet_classification_1997}). Convergence in law
towards each of these distributions is implied
by $L^2(E,\mu)$-convergence of $\Gamma(X_n) - P(X_n)$ towards zero, where $P$ is a polynomial
of degree $0$ (Gaussian distribution), $1$ (Gamma distribution) or $2$ (Beta
distribution). See Table~\ref{tab:targets} for the respective polynomials
and~\cite{ledoux_chaos_2012} for details on how to obtain them in the Gaussian
and Gamma case (the Beta case can be obtained analogously).

\begin{table}
  \centering
  \begin{tabular}{lll}
    target density &$\Gamma$-expression & $\LL$-expression
    \\
    \toprule
    $\frac{\me^{-x^2/2}}{\sqrt{2\pi}}$  & $\Gamma(X) - \lambda_p$ & $\left( \LL
      + 2\lambda_p \Id \right)  H_2(X)$
    \\
    $\frac{1}{\Gamma(\nu)} x^{\nu-1}\me^{-x} \textbf{1}_{(0,\infty)}(x)$    &
    $\Gamma(Y) - \lambda_p Y$ & $\left( \LL +2\lambda_p\Id \right)
    L_2^{(\nu-1)}(Y)$
    \\
    $\frac{1}{B(\alpha,\beta)} x^{\alpha-1}(1-x)^{\beta-1}
    \textbf{1}_{[0,1]}(x)$  & $\Gamma(Y) - \frac{\lambda_p}{\alpha+\beta} Y
    \left( 1 - Y \right)$ & 
    $\left( \LL + 2\lambda_p \frac{\alpha+\beta+1}{\alpha+\beta}   \Id \right)
    P_2^{(\alpha-1,\beta-1)}(1-2Y)$ 
  \end{tabular}
  \caption{$\Gamma$- and corresponding $\LL$-expressions for diffusive
    target distributions. $\LL X=-\lambda_p X$, \, $Y=X+$  mean of target.}
  \label{tab:targets}
\end{table}

Our method now proceeds along the following route. As a first step, we exploit
the identity $2 \Gamma(X_n) = \left( \LL + \lambda_p \Id \right) X_n$,
$\lambda_p$ being the eigenvalue of $X_n$, to obtain an identity of the form
\begin{equation}
  \label{eq:14}
  \Gamma(X_n) - P(X_n) = \left( \LL + a\lambda_p \Id \right) Q(X_n),
\end{equation}
where $a$ is positive real number and $Q$ is a polynomial of degree two (it will
turn out that $Q$ is the second order orthogonal polynomial with respect to the
target measure).
Secondly, we square both sides of~\eqref{eq:14}, integrate and then, imposing
some conditions on $a$, use
Theorem~\ref{General Principle} to remove the square of the integrand on the
right hand side. This allows us to reason that the $L^2(E,\mu)$-convergence of 
$\Gamma(X_n) - P(X_n)$ towards zero is equivalent to the convergence of
\begin{equation}
  \label{eq:16}
  \int_E^{}
  Q(X_n)
  \left( \LL + a\lambda_{p} \Id \right) Q(X_n) \diff{\mu}
\end{equation}
towards zero. As a last step, we notice that due to the derivation property of
$\Gamma$, the integral~\eqref{eq:16} can be expressed as a linear 
combination of the first four moments of $X_n$, denoted by $m_k(X_n)$ in the
sequel ($1 \leq k \leq 4$). This is the content of the next Lemma.

\begin{lemma}
  \label{lem:2}
  Let $X$ be an eigenfunction of $\LL$ with eigenvalue $\lambda_p$, $a \in \R$
  and $Q$ be a polynomial of degree two in one variable. Then it holds that
  \begin{equation}
    \label{eq:17}
  \int_E^{}
  Q(X)
  \left( \LL + a\lambda_{p} \Id \right) Q(X) \diff{\mu}
  =
  \lambda_p
  \int_{\R}
  \left(
    a Q^2(x)
    -
    \frac{\left(Q'(X)\right)^3 X}{3 Q''(X)}
  \right)
  \diff{\mu}.
\end{equation}
\end{lemma}

\begin{remark}
  Note that as the degree of $Q$ is two, $Q'(x)/Q''(x)$ is a polynomial and so
  the integral on the right  
  hand side of~\eqref{eq:17} is always well defined.
\end{remark}

\begin{proof}
  Using the integration by parts formula and the derivation property of
  $\Gamma$, we obtain that
  \begin{align*}
    \int_E^{}
    Q(X) \LL Q(X) \diff{\mu}
    =
    -
    \int_E^{}
    \Gamma(Q(X))
    \diff{\mu}
    &=
    -
    \int_E^{}
    \left( Q'(X) \right)^2 \Gamma(X)
    \diff{\mu}
    \\ &=
    -
    \int_E^{}
    \Gamma \left(
      \frac{\left( Q'(X) \right)^3}{3 Q''(X)}
      ,
      X
    \right)
    \\ &=
    \int_E^{}
    \frac{\left( Q'(X) \right)^3}{3 Q''(X)}
    \LL X
    \diff{\mu}
    \\ &=
    - \lambda_p
    \int_E^{}
    \frac{\left( Q'(X) \right)^3}{3 Q''(X)}
    X
    \diff{\mu},
  \end{align*}
  which yields the stated result.
\end{proof}

Using this method, we now proceed to proof Fourth Moment Theorems for the
caseses where the target measure is Gaussian, Gamma or Beta. Note that our
approach in principle also works for 
other target measures $\mu$, as long as it admits moments of all orders 
(i.e. $\int_{\R} x^n \mu \left( \diff{x} \right)$ is finite for all $n \geq
1$). Indeed, if  
this is the case, one can obtain the corresponding sequence $(P_n)_{n \geq 0}$
of orthogonal polynomials, calculate the (unique) constant $a$ such that
\begin{equation*}
  a P_n^2(x) - \frac{\left( P_2'(x) \right)^2 x}{3 P_2''(x)}
  =
  \sum_{i=1}^{2n} a_i P_n(x)
\end{equation*}
and then use Lemma~\ref{lem:2} to obtain the $\LL$-expression (and by
integration by parts also the $\Gamma$-expression).
For clarity of exposition we present our results in finite dimension, but
everything remains valid in the infinite dimensional setting by a simple limit
procedure. It is also remarkable to note that the constant appearing in the
bounds of the $\Gamma$-expressions by the corresponding moments are independent of the
dimension of the state space.

\subsection{Gaussian approximation}

In order to converge towards a standard Gaussian distribution, we have to
control the quantity $\Gamma(X) - \lambda_p$. The next Theorem gives a precise
bound in terms of moments. 

\begin{theorem}
  \label{theorem:2}
  Let $X$ be a chaos eigenfunction of order $p$ with respect to $-\LL$ and
  assume that $\lambda_{2p} \leq 2\lambda_p$. Then it holds that
  \begin{equation*}
    \int_E^{} \left( \Gamma(X) - \lambda_p \right)^2 \diff{\mu}
    \leq
    \frac{\lambda_p^2}{3}
    \big(
      m_4(X) - 6m_2(X) + 3
    \big).
  \end{equation*}
\end{theorem}

\begin{proof}
  First note that $\Gamma(X) - \lambda_p = \left( \LL + 2\lambda_p \Id\right) \frac{1}{2} H_2(X)$,
  where $H_2(x)= x^2 -1$ is the second Hermite polynomial.
  Thus, by symmetry of the operator $\LL + 2\lambda_p \Id$, we get that 
  \begin{align*}
    \int_E^{} \left( \Gamma(X) - \lambda_p \right)^2 \diff{\mu}
    &=
    \frac{1}{4}
    \int_E^{}
    \left( 
      \left( \LL + 2\lambda_p \Id \right) H_2(X)
       \right)^2
    \diff{\mu}
    \\ &=
    \frac{1}{4}
    \int_E^{}
    H_2(X)
   \left( \LL + 2\lambda_p \Id \right)^2 H_2(X)
   \diff{\mu}.
 \end{align*}
 As $X^2$ and therefore $H_2(X)$ has an expansion on the first $2p$
  eigenspaces and we assume that $\lambda_{2p} \leq 2 \lambda_p$, we can apply
  Theorem~\ref{General Principle} to obtain that
  \begin{equation*}
    \int_E^{} \left( \Gamma(X) - \lambda_p \right)^2 \diff{\mu}
    \leq
    \frac{\lambda_p}{2}
    \int_E^{} H_2(X) \left( \LL + 2\lambda_p \Id \right) H_2(X) \diff{\mu}. 
  \end{equation*}
  Finally we apply Lemma~\ref{lem:2} and obtain that
  \begin{align*}
    \int_E^{} H_2(X) \LL_{2\lambda_p} H_2(X) \diff{\mu}
    &=
    \lambda_p \int_E^{} \left( 2 H_2(X)^2 - \frac{\left( H_2'(X) \right)^3
        X}{3H_2''(X)} \right) \diff{\mu}
    \\ &=
    \lambda_p \int_E^{} \left( 2 \left( X^2 - 1 \right)^2 - \frac{4}{3} X^4
    \right) \diff{\mu}
    \\ &=
    \frac{2\lambda_p}{3}
    \big( m_4(X) - 6 m_2(X) + 3 \big).
  \end{align*}
\end{proof}

\begin{remark}
  Note that the condition $\lambda_{2p} \leq 2 \lambda_p$ is always satisifed if
  $X$ is chaotic in the sense of Ledoux. As a matter of fact,
  Theorem~\ref{theorem:2} shows that the assumed spectral condition (17) of
  Corollary 7 in~\cite{ledoux_chaos_2012} always holds in this case. 
\end{remark}

\begin{remark}
 By exploiting the fact that $X^2-m_2(X)$ is centered, we have
  \begin{equation*}
    \int_E^{} (X^2-m_2(X)) \left( \LL + \lambda_1\Id \right) \left( \LL +
    2\lambda_p \Id\right) (X^2-m_2(X)) \diff{\mu} \leq 0
  \end{equation*}
and proceed as in the proof of Theorem~\ref{General Principle} to get  \begin{equation*}
    4 \int_E^{} \left( \Gamma(X) - \lambda_p m_2(X)\right)^2 \diff{\mu}
    \leq
    \left( 2\lambda_p - \lambda_1 \Id\right)
    \int_E^{}
    (X^2-m_2(X)) \left( \LL + 2\lambda_p \Id\right) (X^2-m_2(X)) \diff{\mu}.
  \end{equation*}
  This yields the better estimate
  \begin{equation*}
    \int_E^{} \left( \Gamma(X) - \lambda_p m_2(X)\right)^2 \diff{\mu}
    \leq
    \left( \frac{\lambda_p^2}{3} - \frac{\lambda_1\lambda_p}{6} \right)
    \big( m_4(X) - 3m_2^2(X) \big).
  \end{equation*}
\end{remark}

\begin{corollary}[Fourth Moment Theorem for Gaussian Approximation]\label{cor:Gauss-app} Let $\{X_n\}_{n \geq 1}$ be a sequence of
  chaos eigenfunctions of order $p$ with respect to the operator
  $-\LL$, bounded in $L^2(E,\mu)$. Then, if $\lambda_{2p} \leq
  2\lambda_{p}$, the following two assertions are equivalent. 
\begin{enumerate}[(i)]
 \item As $n$ tends to infinity, the sequence $\{X_n\}_{n \geq 1}$
converges in distribution to a standard Gaussian distribution.
 \item As $n$ tends to infinity, it holds that $m_{4}(X_n) - 6m_2(X_n) + 3 \to
   0$. 
\end{enumerate}
\end{corollary}

\begin{proof} $(i) \to (ii)$: By hypercontractivity, the fact that
  the 
  sequence $\{X_n\}_{n \geq 1}$ is bounded in $L^2(E,\mu)$ implies that it is also 
bounded in $L^{r}(E,\mu)$ for any $r \geq 1$. Consequently, we obtain that $m_{4}(X_n)
- 6m_2(X_n)+ 3 \to  0$ by the continuous mapping theorem. 

$(ii) \to (i)$: By the integration by parts formula for $\Gamma$, we obtain that
\begin{equation*}
  \int_E^{} e^{\mi \xi X_n} \Gamma(X_n) \diff{\mu}
  =
  \frac{\lambda_p}{i \xi}
  \int_E^{}
   X_n e^{\mi \xi X_n} \diff{\mu}.
\end{equation*} By (ii), we have $\Gamma(X_n) \stackrel{L^2}{\approx}
\lambda_p$, implying that $\int_{E} e^{\mi \xi X_n} \diff{\mu}
\approx \frac{1}{i \xi} \int_{E} X_n e^{i \xi X_n} \diff{\mu}$. For any limit
$\rho$ of any subsequence of $X_n$ we get
\begin{equation}\label{ODE-Guass} \hat\rho (\xi) = - \frac{1}{\xi} \frac{\ud
\hat \rho (\xi)}{\ud \xi},
\end{equation}
where $\hat\rho$ denotes the Fourier transform of $\rho$, and we
conclude the proof by noting that the only solution of the
above differential equation satisfying $\hat \rho(0) = 1$ is given by $\hat \rho
(\xi) = e^{- \frac{\xi^2}{2}}$. 
\end{proof}

\begin{remark} \hfill
  \begin{enumerate}[(i)]
  \item If we additionally assume that $m_2(X_n) = 1$, we can replace
    condition (ii) by $m_4(X_n) - 3 \to 0$.
  \item  To obtain convergence in stronger distances with precise
estimates using Stein's method, we refer to \cite[page. 8]{ledoux_chaos_2012} and
\cite[page. 63]{nourdin_normal_2012}.
  \end{enumerate}

\end{remark}

\subsection{Gamma Approximation}

\begin{theorem}\label{Gamma-aprox} Let $X$ be a chaos eigenfunction with eigenvalue $\lambda_p$ with respect to the operator $-\LL$ such that $
2 \lambda_p \leq \lambda_{2p}$ and set $Y = X+\nu$ for some $\nu > 0$. Then it holds that  
\begin{equation*}
  \int_E^{} \left( \Gamma(Y) - \lambda_p Y \right)^2 \diff{\mu}
  \leq
  \frac{\lambda_p^2}{3}
  \big(
    m_4(X) - 6m_3(X)
   + 6 \left( 1 - \nu \right) m_2(X)
   + 3\nu^2
   \big).
\end{equation*}
\end{theorem}

\begin{proof}
  Using the identities $\LL Y =  - \lambda_p (Y-\nu)$ and
  $2 \Gamma(Y) = \left( \LL + 2\lambda_p \right) (Y - \nu)^2 = \left( \LL + 2\lambda_p \Id
  \right) Y^2 - 2\lambda_p \nu Y $ it is straightforward to verify that
  \begin{equation*}
    \Gamma(Y) - \lambda_p Y
    =
    \left( \LL + 2\lambda_p \Id \right) L_2^{(\nu-1)}(Y),
  \end{equation*}
  where $L_2^{(\nu-1)}(x)= \frac{x^2}{2} - (\nu+1)x + \frac{\nu (\nu+1)}{2}$ is
  the second Laguerre polynomial with parameter $\nu-1$. Due to the chaos
  property of $X$ and the assumption that $\lambda_{2p} \leq 2\lambda_p$, we can
  apply Theorem~\ref{General Principle} to get that
  \begin{equation*}
    \int_E^{} \left(
      \left( \LL + 2\lambda_p \Id \right) L_2^{(\nu-1)}(Y)
    \right)^2 \diff{\mu}
    \leq
    2\lambda_p
    \int_E^{}
    L_2^{(\nu-1)}(Y) \left( \LL + 2\lambda_p \Id \right) L_2^{(\nu-1)}(Y)
    \diff{\mu}.
  \end{equation*}
  Simple calculations after an application of Lemma~\ref{lem:2} now give
  \begin{align*}
    \int_E^{}
    L_2^{(\nu-1)}(Y) &\left( \LL + 2\lambda_p \Id \right) L_2^{(\nu-1)}(Y)
    \diff{\mu}
    \\ &=
    \lambda_p
    \int_E^{}
    \left(
      2 L_2^{(\nu-1)}(Y)^2 - \frac{\left( L_2^{(\nu-1)'}(Y) \right)^3 Y}{3
        L_2^{(\nu-1)''}(Y)} 
    \right)
    \diff{\mu}
    \\ &=
    \lambda_p \left(
      \frac{\nu^2}{2}
      +
      \left( 1 - \nu \right)
      m_2(X)
      -
      m_3(X)
      +
      \frac{m_4(X)}{6}
    \right),
  \end{align*}
  concluding the proof.
\end{proof}

\begin{corollary}[Fourth Moment Theorem for Gamma Approximation] Let $\{X_n\}_{n \geq 1}$ be a sequence bounded in $L^2(E,\mu)$ of chaos eigenfunctions
  of order $p$ with respect to the operator $-\LL$ such that $\lambda_{2p} \leq
  2\lambda_p$, and set $Y_n = X_n + \nu$. Then, the following two 
  assertions are equivalent:
\begin{enumerate}[(i)]
 \item As $n$ tends to infinity, the sequence $\{Y_n\}_{n \geq 1}$
converges in distribution to a Gamma distributed random variable with parameter
$\nu$. 
\item As $n$ tends to infinity, it holds that
  \begin{equation}
    \label{eq:18}
      m_4(X_n) - 6m_3(X_n)
   + 6 \left( 1 - \nu \right) m_2(X_n)
   + 3\nu^2 \to 0.
  \end{equation}
\end{enumerate}
\end{corollary}

\begin{proof}
  The proof can be given analogously to Corollary~\ref{cor:Gauss-app}.
\end{proof}

\begin{remark} \hfill
  \begin{enumerate}[(i)]
  \item If we additionally assume that $m_2(X_n) = \nu$, the moment
    condition~\eqref{eq:18} can be replaced by $m_{4}(X_n) - 6 m_{3}(X_n) -3
    \nu^{2} + 6\nu \to 0$.
  \item Using Stein's method, it is possible to show that the $L^2(E,\mu)$-convergence
    of $\Gamma(X_n+\nu) - \lambda_p (X_n+\nu)$ implies convergence in stronger
    distances. We refer to~\cite[page. 9]{ledoux_chaos_2012} and~\cite{nourdin_noncentral_2009} for details.
  \end{enumerate}
\end{remark}

\subsection{Beta Approximation}
\label{s-sect-this-sect}

\begin{theorem}
  \label{thm:1}
  Let $X$ be a chaos eigenfunction of order $p$ with respect to $-\LL$ and set
  $Y = X + \frac{\alpha}{\alpha+\beta}$, where $\alpha,\beta > 0$.
  Then, if
  \begin{equation}
    \label{eq:1}
    \lambda_{2p} \leq 2\lambda_p \frac{\alpha+\beta+1}{\alpha+\beta},
  \end{equation}
  it holds that
  \begin{multline}
    \label{eq:2}
    \int_E^{} \left( \Gamma(Y) - \frac{\lambda_p}{\alpha+\beta} Y \left( 1-Y \right)\right) \diff{\mu}
    \\ \leq
    \frac{2 (\alpha+\beta+1)\lambda_p^2}{3(\alpha+\beta)}
    \left(
      m_4(X)
      +
      \frac{3(\alpha+1)}{\alpha+\beta+2} m_3(X)
      +
      3 (\alpha+1)^2 m_2(X)
      -
      (\alpha+\beta)
      \left( \frac{\alpha+1}{\alpha+\beta+2} \right)^3
      \right)
  \end{multline}

\end{theorem}

\begin{proof}
  It is straightforward to verify that
  \begin{equation*}
    \Gamma(Y) - \frac{\lambda_p}{\alpha+\beta} Y \left( 1- Y \right) 
    =
    \frac{1}{(\alpha+\beta+1)(\alpha+\beta+2)} (\LL+2\lambda_p \frac{\alpha+\beta+1}{\alpha+\beta}\Id) P_2^{\alpha-1,\beta-1}(1-2Y),
  \end{equation*}  
  where $P_2^{\alpha,\beta}(x)$
  denotes the second Jacobi-polynomial with parameters $\alpha$ and
  $\beta$. Under assumption~\eqref{eq:1}, we can apply Theorem~\ref{General
    Principle} to infer that
  \begin{multline*}
    \int_E^{} \left(
      \Gamma(Y) - \frac{\lambda_p}{\alpha+\beta} Y \left( 1 - Y \right)
    \right)^2 \diff{\mu}
    \\ \leq 2 \lambda_p c_{\alpha,\beta}
    \int_E^{}
    P_2^{(\alpha-1,\beta-1)}(1-2Y) (\LL+2\lambda_p \frac{\alpha+\beta+1}{\alpha+\beta}\Id)
    P_2^{(\alpha-1,\beta-1)}(1-2Y) \diff{\mu},
  \end{multline*}
  where
  \begin{equation*}
    c_{\alpha,\beta} =
    \frac{1}{(\alpha+\beta)(\alpha+\beta+1)(\alpha+\beta+2)^2}.
  \end{equation*}

  The asserted moment expression on the right hand side of~\eqref{eq:2} are
  obtained 
  after an
  application of Lemma~\ref{lem:2} and some tedious calculations.
\end{proof}

\begin{corollary}[Fourth Moment Theorem for Beta Approximation]
  \label{cor:1}
  Let $\{X_n\}_{n \geq 1}$ be a sequence of chaos eigenfunctions of order $p$ with
  respect to the operator $- \LL$, bounded in $L^2(E,\mu)$ and set $Y_n = X_n +
  \frac{\alpha}{\alpha+\beta}$ where $\alpha,\beta >0$. Then, if $\lambda_{2p}
  \leq 2\lambda_p 
  \frac{\alpha+\beta+1}{\alpha+\beta}$, the following two assertions are equivalent.
  \begin{enumerate}[(i)]
  \item As $n$ tends to infinity, the sequence $\{Y_n \}_{n \geq 1}$
    converges in     distribution to a Beta distribution with parameters $\alpha$ and 
    $\beta$.
  \item As $n$ tends to infinity, it holds that
    \begin{equation*}
           m_4(X_n)
      +
      \frac{3(\alpha+1)}{\alpha+\beta+2} m_3(X_n)
      +
      3 (\alpha+1)^2 m_2(X_n)
      -
      (\alpha+\beta)
      \left( \frac{\alpha+1}{\alpha+\beta+2} \right)^3
      \to 0
    \end{equation*}
  \end{enumerate}
\end{corollary}

\begin{proof}
  The proof can be done as in Corollary~\ref{cor:Gauss-app}. What is slightly
  more involved is deriving the 
  characteristic function $\hat{\rho}$ of the weak limit of $\{Y_n\}_{n \geq
    1}$: Using that $\Gamma(Y_n)
  \overset{L^2}{\approx} 
 \frac{\lambda_p}{\alpha+\beta}\lambda_p Y_n (1-Y_n)$ and thus
 \begin{equation*}
   \int_E^{} \mi \xi \me^{\mi \xi Y_n} \Gamma(Y_n) \diff{\mu}
   \approx
   \int_E^{}
   \frac{\lambda_p}{(\alpha+\beta)}
   \mi \xi \me^{\mi \xi Y_n} Y_n (1-Y_n) \diff{\mu},
 \end{equation*}
 we can infer by integration by parts that $\hat{\rho}$ solves the differential
 equation
 \begin{equation*}
      \mi \xi \frac{\mathrm{d}^2}{\diff{\xi^2}} \phi''(\mi \xi) + (\alpha+\beta
      - \mi \xi) \frac{\mathrm{d}}{\diff{\xi}} \phi(\mi \xi) - \alpha \phi(\mi
      \xi) 
      = 0, 
    \end{equation*}
    which is a version of Kummer's equation. The only solution satisfying
    $\phi(0)=\hat{\rho}(0)=1$ and $\frac{\mathrm{d}}{\diff{\xi}} \phi (0) =
    \frac{\mathrm{d}}{\diff{\xi}} \hat{\rho}(0) = \mi
    \frac{\alpha}{\alpha+\beta}$ is
    \begin{equation*}
      \phi(\xi) = M \left( \alpha,\alpha+\beta,\mi \xi \right).
    \end{equation*}
    Here, $M$ denotes Kummer's confluent hypergeometric function, which is
    well-known to be the characteristic function of a Beta distribution
    with parameters $\alpha$ and $\beta$.
\end{proof}

\begin{remark}
  As is the case in the Gaussian and Gamma approximation, one can apply Stein's 
  method in the spirit of~\cite{ledoux_chaos_2012} to prove 
  convergence in Wasserstein distance and derive precise bounds.
\end{remark}


\section{Applications}\label{Appli} In this section, we give concrete examples to how our
main results can be applied to the Ornstein-Uhlenbeck, Laguerre and Jacobi
generators. To be more precise,  we prove that in  
the Wiener/Laguerre/Jacobi diffusion structures, our definition of chaos is
\textbf{always} satisfied in the eigenspaces and that the assumptions of the
Fourth Moment Theorems of the previous sections are valid (in the Jacobi case
under some paramater condition).  
 
\subsection{Wiener Structure}

For $d\geq 1$, denote by $\mu_d$ the 
$d$-dimensional standard Gaussian measure on $\R^d$. It is well known (see for
example~\cite{bakry_gentil_ledoux_2013}), that
$\mu_d$ is the invariant measure of the Ornstein-Uhlenbeck generator, defined
for any test function $\phi$ by 
\begin{equation}\label{O-U} \LL \phi(x)
=\Delta\phi-\sum_{i=1}^{d}x_i \partial_i \phi(x).
\end{equation} Its spectrum is given by $- \N$ and the eigenspaces are of the
form 
\begin{equation*} \textbf{Ker}(\LL+k\Id)=
\left\{\sum_{i_1+i_2+\cdots+i_{d}=k}\alpha(i_1,\cdots,i_{d})\prod_{j=1}^{d}
H_{i_j}(x_j)\right\},
\end{equation*} where $H_n$ denotes the Hermite polynomial of order 
$n$. 
Any eigenfunction $X$ is thus chaotic in the sense of
Definition~\ref{chaos}. Assume now that $X$ is an eigenfunction of 
$\LL$ with eigenvalue $-\lambda_p=-p$. In particular, $X$ is a
multivariate polynomial of degree $p$. Hence $X^2$ is a multivariate
polynomial of degree $2p$. Note that, by expanding $X^2$ over the basis of
multivariate Hermite polynomials $\prod_{j=1}^{d} H_{i_j}(x_j), i_{j}\ge 0$, we
obtain that $X^2$ has a finite expansion over the first eigenspaces of the
generator $\LL$, i.e.
\begin{equation*} X^2 \in \bigoplus_{k=0}^{M} \K(\LL + k
\Id).
\end{equation*}
For degree reasons, we infer that $M=2p$. As a result one 
can see that Theorem~\ref{General Principle} is applicable and thus the
finite-dimensional version of the celebrated Fourth Moment 
Theorem from~\cite{nualart_central_2005} is a consequence of
Corollary~\ref{cor:Gauss-app}. 

\subsection{Laguerre Structure}
\label{s-let-p-geq}

Let $\nu \geq -1$, and $ \pi_{1,\nu}(\diff{x}) =
x^{\nu-1}\frac{\me^{-x}}{\Gamma(\nu)} \textbf{1}_{(0,\infty)}(x) \ud x$ be the Gamma
distribution with parameter $\nu$ on $\R_+$. The associated Laguerre generator is
defined for any test function $\phi$ (in dimension one) by: 
\begin{equation}\label{Lag1} \LL_{1,\nu} \phi(x)= x\phi''(x)+(\nu+1-x)\phi'(x).
\end{equation}
By a classical tensorization procedure, we obtain the Laguerre generator in
dimension $d$ associated to the measure $ \pi_{d,\nu}(\ud x) = \pi_{1,\nu}(\ud x_1) 
\pi_{1,\nu}(\ud x_2) \cdots \pi_{1,\nu}(\ud x_d)$, where $x=(x_1,x_2, \cdots,x_d)$.
\begin{equation}\label{Lag2} \LL_{d,\nu}\phi(x) = \sum_{i=1}^{d}
\Big{(}x_{i} \partial_{i,i}\phi+(\nu+1-x_i)\partial_i \phi\Big{)}
\end{equation}

It is well known that (see for example~\cite{bakry_gentil_ledoux_2013}) that the spectrum of
$\LL_{d,\nu}$ is given by~$-\N$ and moreover that
\begin{equation} \textbf{Ker}(\LL_{d,\nu} + p\Id) =
\left\{\sum_{i_1+i_2+\cdots+i_{d}=p} \alpha(i_1,\cdots,i_{d})\prod_{j=1}^{d}
L^{(\nu)}_{i_j}(x_j)\right\},
\end{equation}
where $L^{(\nu)}_n$ stands for the Laguerre polynomial of order $n$ with parameter $\nu$ which is defined by
$$ L_n^{(\nu)}(x)= {x^{-\nu} \me^x \over n!}{d^n \over dx^n} \left(\me^{-x} x^{n+\nu}\right).$$
Again, we have the following decomposition:
\begin{equation}\label{decompo2} L^2(\R^d,\pi_{d,\nu})=\bigoplus_{p=0}^\infty
\textbf{Ker}(\LL_{d,\nu} + p \Id)
\end{equation}

Similarly in this framework, we see that any eigenfunction $X$ is a chaotic in
the sense of Definition \ref{chaos}. Assume now that $X$ is an eigenfunction of
$\LL_{d,\nu}$ with eigenvalue $-\lambda_p=-p$. In particular, $X$ is a
multivariate polynomial of degree $p$. Therefore,  $X^2$ is a multivariate
polynomial of degree $2p$. Note that by expanding $X^2$ over the basis of
multivariate Laguerre polynomials $\prod_{j=1}^{d} L^{(\nu)}_{i_j}(x_j), i_{j}\ge
0$, we get that $X^2$ has a finite expansion over the first eigenspaces of the
generator $\LL_{d,\nu}$, i.e. 
\begin{equation*} X^2 \in \bigoplus_{p=0}^{M} \K(\LL_{d,\nu}+ p \Id).
\end{equation*} Again, for degree reasons we infer that $M=2p$ and thus
Theorem~\ref{General Principle} is applicable.

\subsection{Beta Structure}

For $\alpha,\beta > -1$, let $\gamma_{\alpha,\beta}(\diff{x}) =
\frac{\Gamma(\alpha+\beta)}{\Gamma(\alpha)\Gamma(\beta)} x^{\alpha-1}
(1-x)^{\beta-1} \textbf{1}_{[0,1]}\diff{x}$ be the beta distribution and choose a
dimension $d \geq 1$. Then, the generator 
$\LL_{\alpha,\beta}$ associated to the measure $\gamma_{d,\alpha,\beta} := \gamma_{\alpha,\beta} \left(
  \diff{x_1} \right) \cdots \gamma_{\alpha,\beta} \left( \diff{x_d} \right)$ is
given by 
\begin{equation*}
  \LL_{\alpha,\beta} \phi(x)
  =
  \LL_{\alpha,\beta} \phi(x_1,\ldots,x_d)
  = \left( \sum_{i=1}^d
    \left(
    x_i(1-x_i) \partial_{ii}^2 + \left(
    \alpha - \left( \alpha+\beta \right)x_i \right) \partial_i
  \right) \right) \phi(x)
\end{equation*}
and its spectrum $S$ is of the form
\begin{equation}
  \label{eq:4}
  S = \left\{ \lambda_{i_1} + \ldots + \lambda_{i_d} \mid  i_j \geq 0, j = 1,\ldots,d \right\},
\end{equation}
where, here and throughout the rest of this section, $\lambda_k = k (k + \alpha+\beta-1)$.
Again, it holds that $L^2(\gamma_{d,\alpha,\beta}) = \bigoplus_{\lambda \in S}
\K \left( \LL_{\alpha,\beta} + \lambda \Id \right)$ and the kernels are given by
\begin{equation}
  \label{eq:5}
  \K \left( \LL_{\alpha,\beta} + \lambda \Id \right)
  =
  \left\{
    \sum_{\substack{i_1,\ldots,i_d \geq 0 \\ \lambda_{i_1} + \ldots +
        \lambda_{i_d} = \lambda }}^{} a(i_1,\ldots,i_d)
    P_{i_1}^{(\alpha-1,\beta-1)}(1-2x_1) \cdots
    P_{i_d}^{(\alpha-1,\beta-1)}(1-2x_d) 
  \right\},
\end{equation}
where $P_n^{(\alpha,\beta)}(x)$ denotes the $n$th Jacobi polynomial, given by
\begin{equation*}
  P_n^{(\alpha,\beta)}(x) = \frac{(-1)^n}{2^n n!} (1-x)^{-\alpha}(1+x)^{\beta}
  \frac{\mathrm{d}^n}{\diff{x^n}} \left( (1-x)^{\alpha}(1+x)^{\beta}(1-x^2)^n \right).
\end{equation*}

We need the following technical Lemma.
\begin{lemma}
  \label{lem:1}
  Let $\lambda \in S$ and
  \begin{equation}
    \label{eq:7}
        X = \sum_{\substack{i_1,\ldots,i_d \geq 0 \\ \lambda_{i_1} + \ldots +
        \lambda_{i_d} = \lambda }}^{} a(i_1,\ldots,i_d)
    P_{i_1}^{(\alpha-1,\beta-1)}(1-2x_1) \cdots P_{i_d}^{(\alpha-1,\beta-1)}(1-2x_d)
    \in \K \left( \LL_{\alpha,\beta} + \lambda \Id
  \right)
  \end{equation}
  Then it holds that
  \begin{equation*}
    X^2 \in \bigoplus_{\eta \leq M} \K \left( \LL_{\alpha,\beta} + \eta \Id \right)
  \end{equation*}
  where
  \begin{equation}
    \label{eq:6}
    M = \max_{\substack{i_1,\ldots,i_d \geq 0 \\ \lambda_{i_1} + \ldots +
        \lambda_{i_d} = \lambda}} 
      \lambda_{2i_1} + \lambda_{2i_2} + \ldots + \lambda_{2i_d}  
  \end{equation}
\end{lemma}

\begin{proof}
  First note that for $a,b \in \N$ we have $\lambda_a + \lambda_b =
  \lambda_{a+b} - 2ab$. By induction, we thus get for $p \geq 2$ and 
  integers $a_1, \dots, a_p \in \N$ that
  \begin{equation}
    \label{eq:10}
    \lambda_{a_1} + \dots + \lambda_{a_p} =
    \lambda_{a_1 + a_2 + \dots + a_p}
    -
    \sum_{\substack{1 \leq k,l \leq p \\ k \neq l}} a_k a_l.
  \end{equation}
  Consequently, it holds for any index $(i_1,\ldots,i_d)$ occurring in the sum on
  the right hand side of~\eqref{eq:7} that $\lambda_{i_1} + \dots +
  \lambda_{i_d} = \lambda_m - \sum_{\substack{1 \leq k,l \leq d \\ k \neq l}}^{}
  i_k i_l$, where $m$ is the degree of $X$. If $(j_1,\dots,j_d)$ is another
  index from this sum, we thus have
  \begin{equation}
    \label{eq:8}
    \sum_{\substack{1 \leq k,l \leq d \\ k \neq l}}^{}
    i_k i_l
    =
    \sum_{\substack{1 \leq j,k \leq d \\ k \neq l}}^{}
    j_k j_l
  \end{equation}
  and, as $\left( \sum_{k=1}^d i_k \right)^2 -
  \left( \sum_{k=1}^d j_k \right)^2 = m^2 - m^2 = 0$,
  \begin{equation}
    \label{eq:9}
    \sum_{k=1}^d i_k^2 = \sum_{k=1}^d j_k^2.
  \end{equation}
  Now observe that the polynomials in the expansion of $X$ with maximum degree
  $2m$ are of the form
  \begin{equation*}
    \prod_{k=1}^d P_{i_k+j_k}^{(\alpha-1,\beta-1)}(1-2x_k),
  \end{equation*}
  corresponding to the eigenvalue
  \begin{equation}
    \label{eq:11}
    \sum_{k=1}^d \lambda_{i_k+j_k} = \lambda_{2m} - \sum_{k \neq l}^{}
    (i_k+j_k)(i_l + j_l),
  \end{equation}
  where we have used the identity~\eqref{eq:10}.
  Due to~\eqref{eq:8}, it holds that
\begin{equation*}
  \sum_{k \neq l}^{}
  (i_k+j_k)(i_l + j_l)
  =
  2  \left( \sum_{k \neq l}^{} i_k i_l
  +
    \sum_{k \neq l} i_k j_l \right)
\end{equation*}
and a straightforward application of the Cauchy-Schwarz inequality together
with~\eqref{eq:9} shows that
\begin{equation*}
  \sum_{\substack{1 \leq k \leq d \\ k \neq l}}^{} i_k j_l
  \geq
  \sum_{\substack{1 \leq k \leq d \\ k \neq l}}^{} i_k i_l.
\end{equation*}
Plugging this into~\eqref{eq:11} yields that $\sum_{k=1}^d \lambda_{i_k+j_k}
\leq \sum_{k=1}^d \lambda_{2i_k}$
and concludes the proof.
\end{proof}

This yields the following corollary.

\begin{corollary}
  \label{cor:2}
  Let $\lambda \in S$ and $X \in \K \left( \LL_{\alpha,\beta} + \lambda \Id
  \right)$. Then, condition~\eqref{eq:1} of Theorem~\ref{thm:1} is satisfied if
  and only if
  $\alpha+\beta \leq 1$.   
\end{corollary}

\begin{proof}
  A straightforward calculation shows that for $p \geq 0$ it holds that
  \begin{equation}
    \label{eq:20}
    2\lambda_p\frac{\alpha+\beta+1}{\alpha+\beta} \geq \lambda_{2p}
  \end{equation}
  if and only if $\alpha+\beta \leq 1$, with equality only for
$\alpha+\beta=1$. Now write $\lambda = \lambda_{i_1} + \dots +
\lambda_{i_p}$ and apply Lemma~\ref{lem:1}.
\end{proof}

\begin{remark} 
Contrarily to the Gamma and Gaussian distribution, the Beta distribution is not
stable under summation. However, if $F_n$ is in the first eigenspace of the
Jacobi 
diffusion generator, we know from above that
\begin{equation*}
  F_n=\sum_{k=0}^\infty a_k(n) X_k,
\end{equation*}
where $X_k\sim \text{Beta}(\alpha,\beta)$. Thus we can use corollary to give
moment conditions for convergence of a linear combination of Beta random
variables towards another Beta random variable. This might be of independent
interest for statisticians. We stress that of course this scenario is not empty,
for instance, one can trivially take $a_1(n)\to 1$ and $a_k(n)\to 0$ for $k \geq
2$. 
\end{remark}

\subsection{Mixed case}
\label{s-here-we-explain}

In this section we prove Theorem \ref{thm:mix} from the introduction, which we
restate here for convenience. 

\begin{thm:mix}
  Let
$\mathcal{X}=\{\pi_{1,\nu}-\nu\,|\,\nu>-1\}\cup\{\mathcal{N}(0,1)\},$
be the set containing all  centered gamma and Gaussian laws.
Let  $\{X_i\}_{i\geq 1}$ be a sequence of independent random variables such that
for all $i\geq 1$, the law of $X_i$ belongs to $\mathcal{X}$. Now choose $d\geq
1$ and let 
\begin{equation*}
  P_n(x)=\sum_{i_1<i_2<\cdots<i_d}
a_n(i_1,\cdots,i_d) x_{i_1}\cdots x_{i_d}
\end{equation*}
be a sequence of multivariate
homogeneous polynomials of degree $d$. Then the two following statements are
equivalent:  
\begin{enumerate}[(i)]
\item  As $n$ tends to infinity, the sequence $\{ P_n(X) \}_{n \ge 1}$ converges
  in distribution towards $\mathcal{N}(0,1)$. 
\item As $n$ tends to infinity, it holds that $\E(P_n(X)^2)\to 1$ and
  $\E(P_n(X)^4)\to 3$. 
\end{enumerate}
\end{thm:mix}

\begin{proof}

By assumption, for each $i\geq 1$,  the law of $X_i$ belongs to $\mathcal{X}$. Then we set $\mathcal{L}_i$ the univariate diffusion generator associated to $X_i$. Then we define $\LL_n$, by the usual tensorization procedure. For instance if $$(X_1,X_2,X_3)\sim\mathcal{N}(0,1)\otimes\mathcal{N}(0,1)\otimes (\gamma(\nu)-\nu),$$ 
then
$$\LL\phi(x,y,z)= \partial_{1,1}\phi+\partial_{2,2} \phi+z \partial_{3,3}\phi-x\partial_1 \phi-y\partial_2 \phi +(\nu+1-z)\partial_3 \phi.$$
One can check that for all $n\ge 1$ the spectrum of $\LL_n$ is $\N$ and
\begin{eqnarray*}
\textbf{Ker}(\LL_n+d ~\Id)&=& \bigoplus_{\substack{i_1<\cdots <i_d\\d_1+\cdots d_d=d}}\textbf{Ker}(\mathcal{L}_{i_1}+d_{i_1}\Id)\otimes\cdots\otimes \textbf{Ker}(\mathcal{L}_{i_d}+d_{i_d} \Id)\\
\end{eqnarray*}
First we claim that $$P_n(x)=\sum_{i_1<i_2<\cdots<i_d} a_n(i_1,\cdots,i_d) x_{i_1}\cdots x_{i_d}\in\textbf{Ker}(\LL_n+d ~\Id).$$
And the degree of $P_n^2$ is $2d$ so that $P_n^2$ may be expanded over
$\bigoplus_{i\leq 2d}\textbf{Ker}(\LL_n+i \Id).$ Consequently, the eigenspaces
of the mixed structures of Wiener and Laguerre are always chaotic in the sense of
Definition \ref{chaos}. Moreover as one easily verifies that
$\lambda_{2p}=2\lambda_p$, we can apply corollaries \ref{Gamma-aprox} and 
\ref{cor:Gauss-app} and obtain Theorem~\ref{thm:mix} as a particular instance.
\end{proof}

\begin{remark}
We could replace the homogeneous sums by a general eigenfunction (i.e. sums of
products of Hermite and Laguerre polynomials) and also mix Wiener, Laguerre
and/or Jacobi generators. In the latter case, one has to impose some additional
technical conditions involving the parameters of the Jacobi generators, as
$\lambda_{2p}\leq 2 \lambda_p$ doesn't hold in general.
 
\end{remark}
\section{Acknowledgements}

We are grateful to Michel Ledoux for very fruitful discussions and comments
as well as for his warm hospitality at the Universit\'e Paul
Sabatier. We also thank Giovanni Peccati for many useful remarks and Ivan
Nourdin for noting that we can improve the constant in Theorem~\ref{theorem:2}.

\bibliography{refs}{}
\bibliographystyle{amsplain}
\end{document}